\documentclass[graybox]{svmult}

\usepackage{type1cm}        
\usepackage{makeidx}         
\usepackage{graphicx}        
\usepackage{multicol}        
\usepackage[bottom]{footmisc}

\usepackage{newtxtext}       %
\usepackage[varvw]{newtxmath}       

\usepackage{wrapfig}
\usepackage{caption}
\makeindex             

\newcommand{\ep}{\varepsilon}

\newcommand{\wt}{\widetilde}
\newcommand{\h}{\quad}
\newcommand{\dis}{\displaystyle}

\newcommand{\bn}{\mathbf{n}}
\newcommand{\cJ}{\mathcal{J}}
\newcommand{\cP}{\mathcal{P}}
\newcommand{\cW}{\mathcal{W}}

\begin{document}

\title*{On sensitivities regarding shape and topology optimization as derivatives on Wasserstein spaces}
\author{Fumiya Okazaki\orcidID{0009-0005-1585-8397}
and\\ Takayuki Yamada\orcidID{0000-0002-5349-6690}
}
\institute{
Fumiya Okazaki \at Institute of Science Tokyo, Tokyo, Japan, \email{okazaki.f.660b@m.isct.ac.jp}
\and Takayuki Yamada \at The University of Tokyo, Tokyo, Japan, \email{t.yamada@mech.t.u-tokyo.ac.jp}
}

%
%
\titlerunning{\scriptsize{On sensitivities regarding shape and topology optimization as derivatives on Wasserstein spaces}}
\maketitle

\abstract*{Each chapter should be preceded by an abstract (no more than 200 words) that summarizes the content. The abstract will appear \textit{online} at \url{www.SpringerLink.com} and be available with unrestricted access. This allows unregistered users to read the abstract as a teaser for the complete chapter.
Please use the 'starred' version of the \texttt{abstract} command for typesetting the text of the online abstracts (cf. source file of this chapter template \texttt{abstract}) and include them with the source files of your manuscript. Use the plain \texttt{abstract} command if the abstract is also to appear in the printed version of the book.}

\abstract{In this paper, we apply the framework of optimal transport to the formulation of optimal design problems. By considering the Wasserstein space as a set of design variables, we associate each probability measure with a shape configuration of a material in some ways. In particular, we focus on connections between differentials on the Wasserstein space and sensitivities in the standard setting of shape and topology optimization in order to regard the optimization procedure of those problems as gradient flows on the Wasserstein space.}

\section{Introduction}\label{Intro}
Shape and topology optimization are optimization problems to determine the material design and have been widely studied from the viewpoints of both mathematical theory and applications. They are formulated as minimization problems for objective functionals defined on a space of sets. Shape optimization refers to the case where the space of sets consists of ones sharing the same topology, which means we only determine the shape of the boundary. On the other hand, topology optimization incorporates the optimization of the configuration of materials including their topology, which deals with the most general setting of material design problems. In this case, the objective functional is defined on a space of sets without topological constraints and our aim is to find the minimizer of the functional to determine the optimal design.

Although the mathematical formulations of them seem to be challenging as we need to deal with the optimization problem on the space of shapes, which is thought to be a complicated infinite dimensional space, there have been several methods to formulate and implement the optimization. Those methods are roughly classified according to how to describe the material. Let $D \subset \mathbb{R}^d$ be a fixed convex domain. We suppose that we are given a function $\Omega \mapsto \mathcal{J}(\Omega)$ defined on a set $\mathcal{O}_{ad}$ composed of subsets in $D$. One standard method to find a minimizer of $\mathcal{J}$ is to introduce the level set function defined on $D$ corresponding to configuration of material $\Omega \in \mathcal{O}_{ad}$ in such a way that
\begin{align*}
\begin{cases}
\phi > 0\ &\text{on}\ \Omega,\\
\phi = 0\ &\text{on}\ \partial \Omega,\\
\phi < 0\ &\text{on}\ D \backslash \Omega.
\end{cases}
\end{align*}
Then we rewrite the objective function as a function of $\phi$. In order to find a minimizer of the objective function, typically we need to calculate a direction in which the objective function decreases. Since in general the objective function is not Fr\'echet differentiable with respect to the level set function, the direction is often found by employing an alternative sensitivity. In the context of shape optimization, the level set function is updated by the Hamilton-Jacobi equation.
In nice cases, the shape derivative can be obtained as a vector field on the boundary of the current shape and it can describe the movement of the frontier. On the other hand, in order to grasp the sensitivity against the change of the topology, the topological derivative has been calculated in some cases in \cite{Sokolowski99} \cite{Garreau01} \cite{Feijoo03}. The topological derivative is a sensitivity against the creation of a small hole inside the current domain. This sensitivity has been incorporated in the process of optimization in \cite{Allaire06} which considered the shape and topological derivative separately. In \cite{Yamada10}, the method to updated the level set function by the reaction-diffusion equation has been introduced and the topological derivative has been used as a substitution for the gradient of objective functions in the usual sense. See Definitions \ref{shapeder} and \ref{topologyder} for details of these sensitivities.

The first method for solving the problem of topology optimization originated from \cite{BK88} was based on the theory of homogenization and the simple isotropic material with penalization method, which is referred to as SIMP method, was introduced in \cite{Bendsoe89}. We simply refer to the latter method as the density method in this paper. In this method, we consider a nonnegative function $\rho \in L^{\infty}(D)$ and regard it as a distribution of material. Then we optimize the functional defined on $L^{\infty}(D)$ instead of the original one defined on the set of shapes. In this case, we can find a descent direction by calculating the Fr\'echet differential on $L^{\infty}(D)$. It is known that in typical cases the differential can be explicitly calculating by the adjoint method (See Example \ref{exElasticity}.)

In this paper, we focus on methods to regard shape and topology optimization as optimization problems on infinite dimensional manifolds which consists of configurations on a fixed domain. Formulations of shape optimization from this viewpoint has been developed recently. In \cite{MM06} \cite{MM07}, several Riemannian metrics on the space of shapes have been investigated. In general, the structure of a Riemannian manifold determines the gradient of functions defined on the space, which leads us to formulate optimization problems. Based on those Riemannian structures, the shape optimization has been regarded as the optimization problem on the space of shapes and the relation between the gradient with respect to the Riemannian metric and the shape derivative of objective functions was clarified in \cite{Arguillere14} \cite{Schulz14} \cite{Welker21}.

On the other hand, we can also choose the Wasserstein space, which consists of probability measures on the base space, as a space which can describe the distribution of materials. By virtue of the Otto calculus, the Wasserstein space can be thought to have a structure of an infinite dimensional Riemannian manifold, namely, we can construct the notion of smooth curves, tangent spaces, differentials and gradient flows on the Wasserstein space. In particular, we can formulate the optimization procedure on the Wasserstein space by the gradient flow associated with a functional defined on the space. The aim of this article is to reformulate shape and topology optimization as optimization problems on Wasserstein spaces and reconstruct sensitivities of objective functions from the differential structure on those spaces. 

This point of view can be seen in \cite{Buttazzo01} which focuses on the shape optimization problem with a mass constraint and writes down the systems which optimal distributions of materials need to satisfy. In \cite{MM07} referred to above, the comparison of the Riemannian metric on the space of shapes and the Wasserstein distance is also mentioned by making the boundary of shapes correspond to the uniform distribution on the boundary. However, it seems that the gradient of functionals on the Wassetstein space has not yet been calculated in the context of the topology optimization. Our formulation is closer to the one in \cite{Buttazzo01} and we incorporate the change of the topology. In Section \ref{opt}, we propose the two kinds of formulation. The first one in Subsection \ref{Formulation1} is to regard the supports of measures as shapes of a material. As readers can see later in Propositions \ref{WassShape1}, \ref{WassShape2} and \ref{WassTop}, in this case we can easily check that the shape and topological derivatives can be realized as directional derivatives or higher-order differentials along some absolutely continuous curves on the Wasserstein space. Our second formulation in Subsection \ref{Formulation2} consists in the density method but we consider updating the distribution of material by the gradient flow on the Wasserstein distance. We will see in Proposition \ref{WassDens} that the direction to update is given by the gradient of the Fr\'echet differential on $L^{\infty}(\mathbb{R}^d)$ in nice cases. Consequently this method to update turned out to be close to the level set method updated by Hamilton-Jacobi equation.

The outline of this paper is as follows. In Section \ref{transport}, we recall some basic settings and facts regarding the theory of optimal transport. We extract only a few contents which seem to be crucial in this paper from the tremendous theory of optimal transport. In most parts of this section, we refer to \cite{AGS21}. In Section \ref{opt}, we suggest two types of formulations of the shape and topology optimization as mentioned above and reconstruct some sensitivities of objective functions as differentials on the Wasserstein space.

\section{Preliminary}\label{transport} 
First we recall the basic setting for optimal transport problems. Denote the set of Borel probability measures on $\mathbb{R}^d$ by $\mathcal{P}(\mathbb{R}^d)$. For a Borel measurable map $\Phi \colon \mathbb{R}^d \to \mathbb{R}^d$ and $\mu \in \mathcal{P}(\mathbb{R}^d)$, we denote the push-forward measure by $\Phi_{\sharp}\mu$, which is defined by
\[
\Phi_{\sharp}\mu (A):= \mu(\Phi^{-1}(A))\ \text{for}\ A\in \mathcal{B}(\mathbb{R}^d).
\]
For $\mu, \nu \in \mathcal{P}(\mathbb{R}^d)$, let $\mathcal{C}(\mu,\nu)$ be the set of couplings of $\mu$ and $\nu$, namely,
\[
\mathcal{C}(\mu,\nu):=\{ \pi \in \mathcal{P}(\mathbb{R}^d\times \mathbb{R}^d) \mid P_{1 \sharp}\pi=\mu,\ P_{2 \sharp}\pi=\nu \},
\]
where $P_i \colon \mathbb{R}^d \times \mathbb{R}^d \to \mathbb{R}^d$ ($i=1,2$) is the projection to $i$-th $\mathbb{R}^d$. For $\mu, \nu \in \mathcal{P}(\mathbb{R}^d)$, the optimal transport problem for the Euclidean metric on $\mathbb{R}^d$ in the sense of Kantorovich is to find
\[
I(\mu, \nu) := \inf \left\{ \int_{\mathbb{R}^d\times \mathbb{R}^d} \frac{|x-y|^2}{2} \, \pi(dxdy) \mid \pi \in \mathcal{C}(\mu, \nu) \right\}
\]
and the minimizer $\pi \in \mathcal{C}(\mu,\nu)$, which can be shown to exist by the lower semi-continuity of the distance function and the fact that $\mathcal{C}(\mu,\nu)$ is tight, convex, weakly closed in $\mathcal{P}(\mathbb{R}^d \times \mathbb{R}^d)$. The minimizer is called the optimal coupling. We recall the celebrated Brenier's theorem regarding specific cases where $\mu \ll dx$, i.e. $\mu$ is absolutely continuous with respect to the Lebesgue measure on $\mathbb{R}^d$. In these cases, there exists a unique optimal plan $\pi$ and it is described as
\begin{align}\label{Brenier}
\pi = (\mathrm{Id} \times \nabla \phi)_{\sharp}\mu
\end{align}
for some convex function $\phi$, where $\mathrm{Id}$ stands for the identity map on $\mathbb{R}^d$. Here we can ignore the zero set in which $\phi$ is not differentiable since the measure $\mu$ is absolutely continuous with respect to $dx$. We let
\[
\mathcal{P}_2(\mathbb{R}^d):=\left\{ \mu \in \mathcal{P}(\mathbb{R}^d) \left| \int_{\mathbb{R}^d} |x|^2 \, \mu(dx) < \infty \right. \right\}
\]
and simply call the space the Wasserstein space in this paper. Define the $L^2$-Wasserstein distance $\mathcal{W}_2$ on $\mathcal{P}_2(\mathbb{R}^d)$ by
\[
\mathcal{W}_2(\mu, \nu):= I(\mu, \nu)^{\frac{1}{2}}\ \text{for}\ \mu,\nu \in \mathcal{P}_2(\mathbb{R}^d).
\]
It is known that $\mathcal{W}_2$ is actually a distance function on $\mathcal{P}_2(\mathbb{R}^d)$. One of the most significant facts on the Wasserstein space $(\mathcal{P}_2(\mathbb{R}^d),\mathcal{W}_2)$ is that it is not just a metric space, but has a kind of an infinite dimensional Riemannian structure. Namely, we can introduce the tangent space at each $\mu \in \mathcal{P}_2(\mathbb{R}^d)$ through considering absolutely continuous curves on $(\mathcal{P}_2(\mathbb{R}^d),\mathcal{W}_2)$, which play a role of smooth curves in ordinary differentiable manifolds. Here we say a curve $\{ \mu_t\}_{t\in [0,T]}$ on $(\mathcal{P}_2(\mathbb{R}^d),\mathcal{W}_2)$ is absolutely continuous if there exists $f\in L^1([0,T])$ such that
\[
\mathcal{W}_2(\mu_s,\mu_t) \leq \int_s^tf(r)\, dr
\]
for all $s,t \in [0,T]$ with $s\leq t$. If $\{ \mu_t\}_{t\in [0,T]}$ is absolutely continuous, for a.e. $t\in [0,T]$ the curve $\mu_t$ admits the metric derivative $|\mu'_t| \in L^1([0,T])$ defined by
\[
|\mu'_t|:= \lim_{\ep \to 0} \frac{\mathcal{W}_2(\mu_t,\mu_{t+\ep})}{|\ep|}.
\]
Absolutely continuous curves on $(\mathcal{P}_2(\mathbb{R}^d),\mathcal{W}_2)$ can be characterized through the continuity equation for measures. First, for simplicity, let $\theta \in W^{1,\infty}(\mathbb{R}^d; \mathbb{R}^d)$, where $W^{1,\infty}(\mathbb{R}^d;\mathbb{R}^d)$ is the $(1,\infty)$-Sobolev space on $\mathbb{R}^d$ with values in $\mathbb{R}^d$. Then the ordinary differential equation
\begin{align*}
\begin{cases}
\frac{d}{dt}\Phi_t(x)&=\theta (\Phi_t(x)),\\
\Phi_0&=\mathrm{Id}
\end{cases}
\end{align*}
admits a global solution for every $x \in \mathbb{R}^d$ since $\theta$ is Lipschitz. For $\rho dx \in \mathcal{P}_2(\mathbb{R}^d)$, we set $\rho_tdx := \Phi_{t\sharp}(\rho dx)$. Then we can easily check that for every $\psi(t,x) \in C^{\infty}_0((0,T)\times \mathbb{R}^d)$, it holds that
\begin{align*}
\int_0^T\int_{\mathbb{R}^d} \psi (t,x)& \partial_t \rho_t(x)\, dxdt\\
&=-\int_0^T\int_{\mathbb{R}^d} \partial_t\psi (t,x)\rho_t (x)\, dxdt \\
&=-\int_0^T\int_{\mathbb{R}^d} \partial_t \psi (t,\Phi_t(x)) \rho (x)\, dxdt\\
&=-\int_0^T\int_{\mathbb{R}^d} \left( \frac{d}{dt} \psi (t,\Phi_t(x)) - \nabla_x \psi (t, x) \cdot \theta (x) \right)\rho(x) \, dxdt\\
&=-\int_0^T\int_{\mathbb{R}^d} \psi (t,x) \mathrm{div}\left(\rho \theta \right)(x)\, dx.
\end{align*}
Thus $\rho_t$ satisfies
\begin{align}\label{conti}
\partial_t\rho_t(x)+\mathrm{div}\left( \rho_t \theta \right) (x)=0,
\end{align}
which is called the continuity equation. In general, the continuity equation for measures in a distributional sense is defined as follows.
\begin{definition}
Let $T>0$, $\{\mu_t\}_{t\in [0,T]}$ a curve on $\mathcal{P}_2(\mathbb{R}^d)$ and $\theta \colon [0,T] \times \mathbb{R}^d \to \mathbb{R}^d$ a time-dependent Borel vector field. We say that the curve $\mu_t$ satisfies the continuity equation with respect to $\theta$ in the sense of distribution if for all $\phi \in C_0((0,T)\times \mathbb{R}^d)$,
\[
\int_0^T\int_{\mathbb{R}^d} \left(\partial_t\phi(t,x) +\nabla_x \phi(t,x) \cdot \theta_t(x) \right)\, d\mu_t dt=0.
\]
\end{definition}
It is known that for every absolutely continuous curve $\{ \mu_t\}_{t\in [0,T]}$ on $\mathcal{P}_2(\mathbb{R}^d)$ with
\begin{align}
\int_0^T|\mu'_t|\, dt<\infty,\label{integrability}
\end{align}
there exists a Borel vector field $\theta \colon [0,T] \times \mathbb{R}^d \to \mathbb{R}^d$ with
\[
\int_0^T \| \theta_t \|_{L^2(\mu_t)}\, dt \leq \int_0^T|\mu'_t|\, dt
\]
such that $\mu_t$ satisfies the continuity equation with respect to $\theta$. Conversely, if a narrowly continuous curve $\{ \mu_t\}_{t\in [0,T]}$ (which means that $\mu_s \to \mu_t$ weakly as a sequence of measures as $s \to t$ for each $t\in [0,T]$) on $\mathcal{P}_2(\mathbb{R}^d)$ satisfies the continuity equation for some vector field $\theta \colon [0,T] \times \mathbb{R}^d \to \mathbb{R}^d$ with
\[
\int_0^T \| \theta_t \|_{L^2(\mu_t)}\, dt < \infty,
\]
then $\{ \mu_t \}_{t\in [0,T]}$ is absolutely continuous for $\mathcal{W}_2$. Note that the uniqueness of velocity vector fields satisfying the continuity equation is not guaranteed for every absolutely continuous curve $\mu_t$ without any conditions for velocity vector fields. Thus in order to determine velocity vector fields associated with absolutely continuous curves, we impose an additional condition for the energy. Let $L^2(\mu, \mathbb{R}^d; \mathbb{R}^d)$ be $\mathbb{R}^d$-valued $L^2$-space with respect to the measure $\mu \in \mathcal{P}_2(\mathbb{R}^d)$. Then it is known that we can take a unique $\theta_t \colon \mathbb{R}^d \to \mathbb{R}^d$ such that $\theta_t \in L^2(\mu_t, \mathbb{R}^d; \mathbb{R}^d)$ and for a.e. $t \in [0,T]$,
\begin{align}\label{projection}
\| \theta_t + \eta \|_{L^2(\mu_t)} \geq \| \theta_t\|_{L^2(\mu_t)} \ \text{for all}\ \eta_t \in L^2(\mu_t, \mathbb{R}^d; \mathbb{R}^d) \ \text{with}\ \mathrm{div}(\mu_t \eta_t)=0. 
\end{align}
If we set
\begin{align}\label{tangent}
T_{\mu}\mathcal{P}_2(\mathbb{R}^d):= \overline{\{ \nabla \phi \mid \phi \in C_0^{\infty}(\mathbb{R}^d) \}}^{L^2(\mu)},
\end{align}
then $\theta_t$ satisfies \eqref{projection} if and only if $\theta_t \in T_{\mu_t}\mathcal{P}_2(\mathbb{R}^d)$ for a.e. $t\in [0,T]$. Thus we can regard the space defined by \eqref{tangent} as the tangent space at $\mu$ and velocity vector fields associated with absolutely continuous curves with \eqref{integrability} can be described as gradients of potentials $\phi \colon \mathbb{R}^d \to \mathbb{R}$.

\begin{definition}
Let $\wt{\cJ} \colon \cP_2(\mathbb{R}^d) \to \mathbb{R}$ be a function. The tangent vector $\nabla^{\cW}\wt{\cJ}(\mu) \in T_{\mu}\cP_2(\mathbb{R}^d)$ is called the gradient of $\wt{\cJ}$ at $\mu$ if for any absolutely continuous curve $\mu_t$ on $\cP_2(\mathbb{R}^d)$ satisfying the continuity equation with a velocity vector field in the form of $\nabla \phi_t$, it holds that
\[
\left( \frac{d}{dt}\right)_{t=0} \wt{\cJ}(\mu_t)=\langle \nabla^{\cW}\wt{\cJ}(\mu), \nabla \phi_0 \rangle_{L^2(\mu)}. 
\]
\end{definition}
Following this definition of the gradient of a function $\wt{\cJ}$, we can derive the gradient flow
\[
\partial_t \rho_t = \mathrm{div}(\rho_t \nabla^{\cW}\cJ(\rho_t)).
\]
Roughly speaking, the gradient flow means that we find the decent direction of $\wt{\cJ}$ as a vector field and transport the sand pile $\rho dx$ following the flow of the vector field (Fig \ref{figupdate}).
\begin{figure}[h]
\centering
\includegraphics[width=7.0cm]{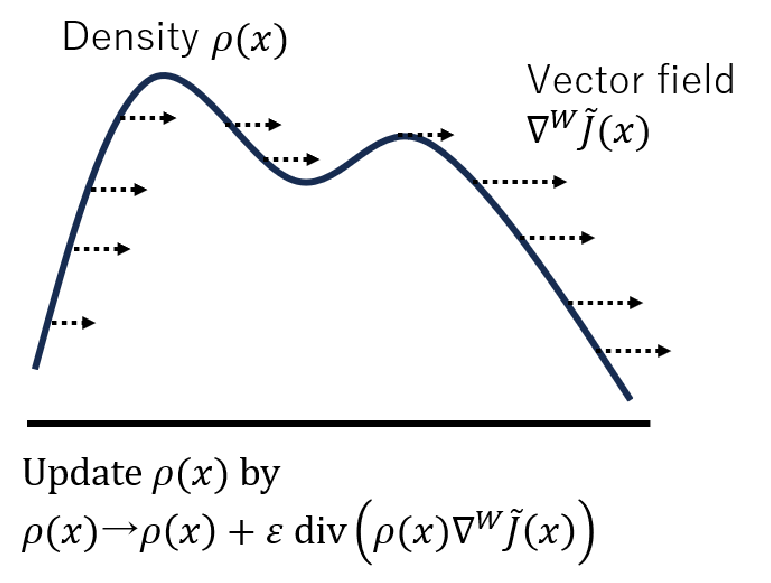}
\caption{}
\label{figupdate}
\end{figure}
\section{Sensitivities}\label{opt}
To begin with, we recall the basic setting for the shape and topology optimization.
We denote the set of all open subsets in $\mathbb{R}^d$ by $\mathcal{O}$. Let $\mathcal{O}_{ad}$ be a subfamily of $\mathcal{O}$. We consider a function $\mathcal{J}\colon \mathcal{O}_{ad} \to \mathbb{R}$. In many cases, $\mathcal{J}$ is of the form as follows: We suppose that we are given a system described by a PDE on an open subset $\Omega$. Let $H$ be a Hilbert space in which solutions of the system live. Let $a_{\Omega}$ be a bounded coercive symmetric bilinear form on $H$ depending on $\Omega \in \mathcal{O}_{ad}$. In the same way, let $l_{\Omega}$ be a bounded linear functional on $H$ depending on $\Omega \in \mathcal{O}_{ad}$. Then there exists a unique $u_{\Omega} \in H$ satisfying
\begin{align}\label{governingEq}
a_{\Omega}(u_{\Omega},v)=\langle l_{\Omega}, v \rangle \ \text{for all}\ v\in H
\end{align}
by Lax-Milgram's theorem. Let $J \colon \mathcal{O}_{ad} \times H \to \mathbb{R}$ and suppose that the functional $\mathcal{J} \colon \mathcal{O}_{ad} \to \mathbb{R}$ is given by
\[
\mathcal{J}(\Omega)=J(\Omega, u_{\Omega}).
\]
In order to apply the gradient descent method to the functional $\mathcal{J}$, we need to consider the perturbation of a domain in $\mathcal{O}_{ad}$. One natural way to give a perturbation is to take a vector field $\theta \colon \mathbb{R}^d \to \mathbb{R}^d$ and consider the image $(\mathrm{Id} + \theta)(\Omega)$. Here this procedure needs the condition for $\mathcal{O}_{ad}$ and a space of vector fields $\Theta_{ad}$ such that $\mathcal{O}_{ad}$ contains all the perturbated domain $(\mathrm{Id} + \theta)(\Omega)$ for $\Omega \in \mathcal{O}_{ad}$ and all small enough $\theta$ in terms of a suitable norm on $\Theta_{ad}$. For instance, if we suppose that $\mathcal{O}_{ad}$ contains $C^k$-domains, one can set
\begin{align*}
\Theta_{ad}&=C^{k,\infty}(\mathbb{R}^d;\mathbb{R}^d)\\
&=\left\{ \theta \in C^k(\mathbb{R}^d; \mathbb{R}^d)\ ;\ \| \theta \|_{C^k}<\infty  \right\}
\end{align*}
as a set of vector fields, where
\[
\| \theta \|_{C^k}= \sup_{|\alpha|\leq k}\| \partial_{\alpha} \theta \|_{L^{\infty}(\mathbb{R}^d)}
\]
and $\alpha$ stands for a multi index (See Remark 4.1 of \cite{Allaire21}). Now we recall the sensitivity of the functional regarding the above perturbation.
\begin{definition}\label{shapeder}
A function $\mathcal{J} \colon \mathcal{O}_{ad} \to \mathbb{R}$ is said to be shape-differentiable at $\Omega \in \mathcal{O}_{ad}$ if the functional 
\[
\theta \mapsto \mathcal{J}((\mathrm{Id}+\theta)(\Omega))
\]
defined on $W^{1,\infty}(\mathbb{R}^d ; \mathbb{R}^d)$ is Fr\'echet differentiable at $\theta=0$, namely, there exists a bounded linear functional $d_S\mathcal{J}(\Omega)$ on $W^{1,\infty}(\mathbb{R}^d;\mathbb{R}^d)$ such that
\[
\mathcal{J}((\mathrm{Id}+\theta)(\Omega))=\mathcal{J}(\Omega)+\langle d_S\mathcal{J}(\Omega), \theta \rangle + o(\theta),
\]
where
\[
\frac{o(\theta)}{\| \theta \|_{W^{1,\infty}(\mathbb{R}^d ; \mathbb{R}^d)}} \to 0\ \text{as}\ \| \theta \|_{W^{1,\infty}(\mathbb{R}^d ; \mathbb{R}^d)} \to 0.
\]
In this case the functional $d_S\mathcal{J}(\Omega)$ is called the shape derivative of $\mathcal{J}$ at $\Omega$.
\end{definition}
In the process of the shape optimization, we need to calculate the shape derivative of the objective function and find a vector field which decreases the value of it by the deformation of the domain. It is known that in typical nice cases the shape derivative can be obtained in the form of
\[
\langle d_S\mathcal{J}(\Omega), \theta \rangle = \int_{\partial \Omega} g \langle \theta, \mathbf{n}\rangle \, d\mathrm{vol}_{\partial \Omega},
\]
where $g \in L^1(\partial \Omega)$ and $\mathbf{n}$ is the outward normal unit vector field.

Since the above deformation is homeomorphic, theoretically the topology of the domain does not change in the process of optimization. Thus in order to incorporate changes of topology and implement the optimization including the the topology of the material, we need to employ other sensitivities. Here we recall the topological sensitivity defined through the perturbation by the creation of small holes.
\begin{definition}\label{topologyder}
Let $B \subset \mathbb{R}^d$ be an open subset. A function $\mathcal{J} \colon \mathcal{O}_{ad} \to \mathbb{R}$ is said to admit a topological derivative at $\Omega \in \mathcal{O}_{ad}$ if there exists a strictly increasing continuous function $r \colon [0,\infty) \to [0,\infty)$ such that the limit
\[
\lim_{\ep \to 0} \frac{\mathcal{J}(\Omega \backslash \overline{(z + \ep B)})-\mathcal{J}(\Omega)}{r(\ep)}
\]
exists for every $z\in \Omega$. The limit is called the topological derivative for the hall $B$ with a rate $r$ and denoted by $d_T\mathcal{J}(\Omega,B,z)$.
\end{definition}
The rate $r$ depends on the functional, but we focus on the typical case where $B$ is the unit ball $B_1(0)$ centered at the origin in $\mathbb{R}^d$ and $r(\ep)=|B_{\ep}(0)|$. In this case we simply call the function $z\mapsto \mathcal{J}(\Omega, B_1(0),|B_{\cdot}(0)|)$ the topological derivative at $\Omega$ and denote it by $d_T\mathcal{J}(\Omega,z)$.

\subsection{Formulation based on the support of measures}\label{Formulation1}
We reformulate the above setting of shape optimization in view of the optimization on Wasserstein spaces. In this framework, we regard probability measures on $\mathbb{R}^d$ as distributions of materials. For $\mu \in \mathcal{P}(\mathbb{R}^d)$, let
\[
\Omega_{\mu}:=\mathrm{Int}(\mathrm{supp}[\mu]),
\]
where $\mathrm{Int}$ stands for the interior. We define the projection $\pi \colon \mathcal{P}_2(\mathbb{R}^d) \to \mathcal{O}$ by
\[
\pi (\mu):= \Omega_{\mu}.
\]
Obviously $\pi$ is surjective. In fact, each non-empty $\Omega \in \mathcal{O}$ has strictly positive Lebesgue measure. Thus if we let $\rho$ be the probability density of $d$-dimensional Gaussian distribution, then
\[
\pi \left(\left( \int_{\Omega}\rho dm \right)^{-1}\mathbf{1}_{\Omega}\rho \cdot m \right) =\Omega,
\]
where $m$ is the Lebesgue measure on $\mathbb{R}^d$ and $\mathbf{1}_{\Omega}$ is the characteristic function of $\Omega$. As for the case where $\Omega=\emptyset$, by taking the delta distribution at $0$, we have
\[
\pi(\delta_0)=\emptyset.
\]
We set $\mathcal{P}_{ad}(\mathbb{R}^d)=\pi^{-1}(\mathcal{O}_{ad})$, namely,
\[
\mathcal{P}_{ad}(\mathbb{R}^d):=\{ \mu \in \mathcal{P}_2(\mathbb{R}^d) \mid \Omega_{\mu}\in \mathcal{O}_{\mathrm{ad}} \}.
\]
For a function $\mathcal{J} \colon \mathcal{O}_{ad} \to \mathbb{R}$, we set
\[
\wt{\mathcal{J}}:=\mathcal{J} \circ \pi \colon \mathcal{P}_{ad}(\mathbb{R}^d) \to \mathbb{R}.
\]
\begin{proposition}\label{WassShape1}
Let $\mu \in \mathcal{P}_{ad}(\mathbb{R}^d)$. Assume that a functional $\mathcal{J}$ defined on a family $\mathcal{O}_{ad}$ is shape-differentiable at $\Omega_{\mu}$. Then for each $\theta \in W^{1, \infty}(\mathbb{R}^d; \mathbb{R}^d)$, the curve $t\mapsto (\mathrm{Id}+t\theta)_{\sharp}\mu$ is in $\mathcal{P}_{ad}(\mathbb{R}^d)$ for small $t>0$ and
\[
t \mapsto \wt{\mathcal{J}}((\mathrm{Id}+t\theta)_{\sharp}\mu)
\]
is differentiable at $t=0$. Moreover the differential is equal to $\langle d_S \mathcal{J}(\Omega_{\mu}), \theta \rangle$.
\end{proposition}
\begin{proof}
Note that for $\theta$ with its norm in $\Theta_{ad}$ small enough, the map $\mathrm{Id}+\theta$ is a diffeomorphism. In particular, we have
\[
\mathrm{supp}\left[ (\mathrm{Id}+\theta)_{\sharp}\mu \right] =(\mathrm{Id}+\theta)(\mathrm{supp}[\mu]).
\]
Thus the claim is obvious since
\begin{align*}
\frac{1}{t}\left( \wt{\mathcal{J}}((\mathrm{Id}+\theta)_{\sharp}\mu)-\wt{\mathcal{J}}(\mu) \right) &=\frac{1}{t}\left( \mathcal{J}((\mathrm{Id}+\theta)\Omega)-\mathcal{J}(\Omega) \right).
\end{align*}
\end{proof}

\begin{proposition}\label{WassShape2}
Let $\mu \in \mathcal{P}_{ad}(\mathbb{R}^d)$. Assume that a functional $\mathcal{J}$ defined on a family $\mathcal{O}_{ad}$ is shape-differentiable at $\Omega_{\mu}$. Let $\mu_t$ satisfy the continuity equation with respect to the vector field $\nabla \phi$ for some $\phi \in C_0^{\infty}(\mathbb{R}^d)$. Then $\mu_t \in \mathcal{P}_{ad}(\mathbb{R}^d)$ for small $t>0$ and
\[
t \mapsto \wt{\mathcal{J}}(\mu_t)
\]
is differentiable at $t=0$. Moreover, the differential is equal to $\langle d_S \mathcal{J}(\Omega_{\mu}), \nabla \phi \rangle$.
\end{proposition}
\begin{proof}
The solution of the continuity equation is given by $\mu_t=\Phi_{t \sharp}\mu$, where $\Phi_t \colon \mathbb{R}^d \to \mathbb{R}^d$ is the solution of
\[
\frac{d \Phi_t}{dt}(x)=\nabla \phi(\Phi_t(x)).
\]
Since $\Phi_t$ is a diffeomorphism for every $t>0$, we have
\[
\mathrm{supp}[\mu_t]=\Phi_t(\mathrm{supp}[\mu]).
\]
Therefore,
\[
\Omega_{\mu_t}=\Phi_t(\Omega_{\mu}) \in \mathcal{O}_{ad}.
\]
If we set 
\[
\dis \theta_t:=\frac{1}{t}\left( \Phi_t-\mathrm{Id} \right)=\frac{1}{t}\int_0^t \nabla \phi(\Phi_s(x))\, ds,
\]
then the map $\theta_t$ is a smooth vector field and satisfies
\[
(\mathrm{Id}+t\theta_t)(\Omega_{\mu})=\Phi_t(\Omega_{\mu}).
\]
Moreover, it holds that
\[
\| \theta_t \|_{W^{1,\infty}(\mathbb{R}^d; \mathbb{R}^d)} \leq \| \nabla \phi \|_{C^2(\mathbb{R}^d; \mathbb{R}^d)}\exp \left(t_0 \|\nabla \phi \|_{C^2(\mathbb{R}^d;\mathbb{R}^d)}\right),
\]
for all $t\in [0,t_0]$, where we applied the estimate
\[
\sup_{0 \leq t \leq t_0}|\nabla \Phi_t(x)| \leq \exp \left(t_0 \|\nabla \phi \|_{C^2(\mathbb{R}^d;\mathbb{R}^d)}\right)
\]
which can be obtained from
\[
|\nabla \Phi_t|(x) \leq 1+ \|\nabla \phi \|_{C^2(\mathbb{R}^d;\mathbb{R}^d)} \int_0^t |\nabla \Phi_s|(x)\, ds.
\]
and Gronwall's inequality. Therefore, we have
\begin{align}
t \theta_t \to 0 \ \text{as}\ t \to 0 \ \text{in}\ W^{1,\infty}(\mathbb{R}^d; \mathbb{R}^d). \label{thetaW}
\end{align}
Since $\mathcal{J}$ is shape differentiable at $\Omega$, it holds that
\begin{align}
\wt{\mathcal{J}}(\mu_t)&= \mathcal{J}(\Omega_{\mu_t}) \nonumber \\
&= \mathcal{J}((\mathrm{Id}+t\theta_t)(\Omega)) \nonumber \\
&= \mathcal{J}(\Omega)+t \langle d_S\mathcal{J}(\Omega), \theta_t \rangle + h(t\theta_t) \nonumber \\
&= \wt{\mathcal{J}}(\mu)+t \langle d_S\mathcal{J}(\Omega), \theta_t \rangle + h(t\theta_t),\label{thetaSD}
\end{align}
where $h \colon W^{1,\infty}(\mathbb{R}^d; \mathbb{R}^d)$ is a functional satisfying
\[
\frac{h(\theta)}{\| \theta \|_{W^{1,\infty}(\mathbb{R}^d ; \mathbb{R}^d)}} \to 0\ \text{as}\ \| \theta \|_{W^{1,\infty}(\mathbb{R}^d ; \mathbb{R}^d)} \to 0.
\]
Thus the map $t \mapsto \wt{\mathcal{J}}(\mu_t)$ is differentiable at $t=0$ and we have
\[
\left( \frac{d}{dt} \right)_{t=0}\wt{\mathcal{J}}(\mu_t)=\langle d_S\mathcal{J}(\Omega), \nabla \phi \rangle
\]
by \eqref{thetaW} and \eqref{thetaSD}.
\end{proof}
Next we interpret the topological derivative as a differential on the Wasserstein space. Unlike the case of shape differential, we need to take a vector field with singularities to incorporate the change of the topology.
In this article, we focus on the fact that tangent spaces on $\mathcal{P}_2(\mathbb{R}^d)$ contains vector fields with singularities. In Proposition \ref{WassTop} below, we describe the topological derivative as a higher-order derivative of functionals lifted to the Wasserstein space in the direction of a vector field with a singular point.
\begin{proposition}\label{WassTop}
Assume that a functional $\mathcal{J} \colon \mathcal{O}_{ad} \to \mathbb{R}$ admits a topological derivative at $\Omega$. For a fixed interior point $x_0\in \Omega$, let $\ep >0$ be a positive number such that $\overline{B_{2\ep}}(x_0) \subset \Omega$. Let $\psi \in C_0^{\infty}(\mathbb{R}^d)$ be a rotationally symmetric function centered around $x_0$ satisfying
\begin{align*}
\begin{cases}
\psi =1\ &\text{on}\ \overline{B_{\ep}(x_0)},\\
\psi =0\ &\text{on}\ \overline{B_{2\ep}(x_0)^c},\\
0\leq \psi \leq 1\ &\text{on}\ \mathbb{R}^d.
\end{cases}
\end{align*}
Let $\phi=|x-x_0|$ and $\theta$ a vector field on $\mathbb{R}^d \backslash \{ x_0 \}$ given by $\nabla (\psi \phi)$. Then $\theta \in T_{\mu}\mathcal{P}_2(\mathbb{R}^d)$. Moreover, if we let $\Phi_t \colon \mathbb{R}^d \backslash \{ x_0\} \to \mathbb{R}^d$ be the one-parameter transformation generated by $\theta$ and set $\mu_t:=\Phi_{t\sharp}\mu$, then
\[
\frac{\alpha_d}{d!}\left( \frac{d}{dt} \right)^d_{t=0}\wt{\mathcal{J}}(\mu_t)=d_T\mathcal{J}(\Omega),
\]
where $\alpha_d>0$ is the constant determined by $|B_r(0)|=\alpha_d r^d$.
\end{proposition}
\begin{proof}
We suppose that $\psi$ is written as $\psi(x) = \eta (|x-x_0|)$ for some $\eta \colon [0,\infty) \to [0,1]$. Then
\[
\theta(x) = \eta'(|x-x_0|)(x-x_0) + \psi(x) \frac{x-x_0}{|x-x_0|}.
\]
From this, we can confirm that
\begin{align*}
\Omega_{\mu_t}&=\Phi_t(\Omega_{\mu})\\
&=\Omega_{\mu}\, \backslash \, \overline{B_t(x_0)}
\end{align*}
for small enough $t>0$. Therefore, we have
\[
\wt{\mathcal{J}}(\mu_t)-\wt{\mathcal{J}}(\mu)=\alpha_d t^d d_T\mathcal{J}(\Omega_{\mu}) + o(t^d)
\]
by the definition of the topological derivative. This completes the proof.
\end{proof}
\begin{remark}
In this formulation, we can interpret shape and topological derivatives as some differentials on the Wasserstein space. However, we still need to deal with those two derivatives separately in order to incorporate the change of topology. In the phase of the implementation, this method can be reduced to a similar procedure to the one in \cite{Allaire06}.
\end{remark}
\subsection{Formulation based on the density method}\label{Formulation2}
In this section, we consider another formulation of shape and topology optimization on Wasserstein spaces. Let $n \in \mathbb{N}$ and $J \colon H^1(\mathbb{R}^d; \mathbb{R}^n) \to \mathbb{R}$ a $C^2$ functional in the sense of the Fr\'echet differential. In a similar way to Subsection \ref{Formulation1}, we suppose that we are given a bounded coercive bilinear form $a_{\rho}$ on $H^1(\mathbb{R}^d; \mathbb{R}^n)$ and a bounded linear functional $l_{\rho}$ on $H^1(\mathbb{R}^d; \mathbb{R}^n)$ for each nonnegative $\rho \in L^{\infty}(\mathbb{R}^d)$. We denote by $u_{\rho}$ a function in $H^1(\mathbb{R}^d; \mathbb{R}^n)$ which satisfies
\[
a_{\rho}(u,\phi)=\langle l_{\rho}, \phi \rangle
\]
for all $\phi \in H^1(\mathbb{R}^d; \mathbb{R}^n)$. Then we consider an objective function given by $\mathcal{J}(\rho):=J(\rho, u_{\rho})$.
This is a typical formulation of the density method for the topology optimization. We consider the pull-back of the objective function to the Wasserstein space. We set
\[
\mathcal{P}_{ac}^{\infty}(\mathbb{R}^d):=\{ \rho dx \in \mathcal{P}_2(\mathbb{R}^d) \mid \rho \in C^{\infty}(\mathbb{R}^d) \}.
\]
For every $\rho dx \in \mathcal{P}_{ac}^{\infty}(\mathbb{R}^d)$ the function $\rho$ decays at infinity since $\rho$ is integrable and consequently $\rho \in L^{\infty}(\mathbb{R}^d)$. Thus we can define the embedding $\iota \colon \mathcal{P}_{ac}^{\infty}(\mathbb{R}^d) \to L^{\infty}(\mathbb{R}^d)$ by
\[
\iota(\rho dx)=\rho.
\]
Then we can define $\wt{\mathcal{J}} \colon \mathcal{P}_{ac}^{\infty}(\mathbb{R}^d) \to \mathbb{R}$ by
\[
\wt{\mathcal{J}} := \mathcal{J} \circ \iota.
\]
\begin{proposition}\label{WassDens}
Let $\rho dx \in \mathcal{P}_{ac}^{\infty}(\mathbb{R}^d)$. Assume that the functional $\mathcal{J}$ is Fr\'echet differentiable at $\rho$ in $L^{\infty}(\mathbb{R}^d)$ and the differential can be written through
\[
\cJ(\rho + \bar{\rho})=\cJ(\rho) + \int_{\mathbb{R}^d} F(x) \bar{\rho}(x)\, dx + h(\bar{\rho})
\]
for some $F \in C^{\infty}(\mathbb{R}^d)$, where $h \colon L^{\infty}(\mathbb{R}^d)\to \mathbb{R}$ satisfies
\[
\frac{h(\bar{\rho})}{\| \bar{\rho} \|_{L^{\infty}}} \to 0\ \text{as}\ \bar{\rho} \to 0\ \text{in}\ L^{\infty}(\mathbb{R}^d)
\]
Let $\mu_t=\rho_tdx$ be an absolutely continuous curve on $(\mathcal{P}_2(\mathbb{R}^d),\mathcal{W}_2)$ satisfying the continuity equation with respect to a vector field $\nabla \phi$ for $\phi \in C_0^{\infty}(\mathbb{R}^d)$.
Then $t\mapsto \wt{\mathcal{J}}(\mu_t)$ is differentiable at $t=0$ and
\[
\left( \frac{d}{dt} \right)_{t=0} \wt{\mathcal{J}}(\mu_t)=\langle \nabla F, \nabla \phi \rangle_{L^2(\rho)},
\]
where $\rho=\rho_0$.
\end{proposition}
\begin{proof}
First, we note that the solution of the continuity equation $\rho_t$ is given by the density of $\Phi_{t\sharp}(\rho dx)$, where $\Phi_t$ is the flow generated by $\nabla \phi$ and the density is written explicitly as
\[
\rho_t(x):=\left( \frac{\rho}{\mathrm{det}(\nabla \Phi_t)} \right)\circ \Phi_t^{-1}(x).
\]
Thus there exists $t_0>0$ such that $(t,x) \mapsto \rho_t(x)$ is $C^{\infty}$ on $[0,t_0]\times \mathbb{R}^d$. Moreover, we can take constants $C_0=C_0(t_0,\phi)>0$ and $C_1=C_1(t_0,\phi)$ in such a way that it holds that
\begin{align*}
|\rho_t| &\leq C_0 |\rho|,\\
|\nabla \rho_t| &\leq C_1 |\nabla \rho|
\end{align*}
on $[0,t_0] \times \mathbb{R}^d$. Since
\[
\mathrm{div}(\rho_t \nabla \phi)=\nabla \rho_t \cdot \nabla \phi + \rho_t \Delta \phi,
\]
we have
\begin{align}\label{supdiv}
\sup_{0\leq t \leq t_0}|\mathrm{div}(\rho_t \nabla \phi)| \leq C_2 \| \rho \|_{C^1} \| \phi \|_{C^2},
\end{align}
where $C_2=\max \{ C_0,C_1 \}$. On the other hand, by the continuity equation, we also have
\[
\rho_t= \rho - \int_0^t \mathrm{div}(\rho_s \nabla \phi) ds.
\]
This yields
\[
|\rho_t - \rho| \leq tC_2 \| \rho \|_{C^1} \| \phi \|_{C^2}.
\]
Therefore, we have
\begin{align*}
\wt{\mathcal{J}}(\mu_t)&=\cJ(\rho_t)\\
&=\cJ(\rho)-\int_{\mathbb{R}^d} F(x) \int_0^t \mathrm{div}(\rho_s \nabla \phi)\, dsdx + h\left(- \int_0^t \mathrm{div}(\rho_s \nabla \phi) ds \right) \\
&=\wt{\mathcal{J}}(\mu) + \int_0^t \int_{\mathbb{R}^d} \nabla F(x) \cdot \nabla \phi(x)\ \rho_s(x)\, dxds +h\left(- \int_0^t \mathrm{div}(\rho_s \nabla \phi) ds \right) \\
&=\wt{\mathcal{J}}(\mu) + t\int_{\mathbb{R}^d} \nabla F(x) \cdot \nabla \phi(x)\, \rho (x) dx\\
&\h + \int_{\mathbb{R}^d} \nabla F(x) \cdot \nabla \phi(x)\int_0^t \int_0^s \mathrm{div}(\rho_r \nabla \phi)\, dr ds dx + h\left( - \int_0^t \mathrm{div}(\rho_s \nabla \phi) ds \right),
\end{align*}
where we applied
\[
\int_0^t \rho_s(x) ds = t \rho(x) + \int_0^t \int_0^s \mathrm{div}(\rho_r \nabla \phi)\, dr ds 
\]
in the fourth equality. By assumption and \eqref{supdiv}, it holds that
\[
\lim_{t\to 0} \frac{1}{t}h\left( - \int_0^t \mathrm{div}(\rho_s \nabla \phi) ds \right)=0.
\]
In a similar way, by \eqref{supdiv}, we have
\begin{align*}
&\left| \frac{1}{t}\int_{\mathbb{R}^d} \nabla F(x) \cdot \nabla \phi(x)\int_0^t \int_0^s \mathrm{div}(\rho_r \nabla \phi)\, dr ds dx \right| \\
&\leq \frac{1}{2} t C_2\|\rho \|_{C^1}\|\phi \|_{C^2} \int_{\mathrm{supp}[\phi]} \left| \nabla F(x) \cdot \nabla \phi(x)\right| \, dx\\
& \to 0\ \text{as}\ t\to 0.
\end{align*}
Therefore, $t\mapsto \wt{\mathcal{J}}(\mu_t)$ is differentiable at $t=0$ and it holds that
\[
\left( \frac{d}{dt} \right)_{t=0} \wt{\mathcal{J}}(\mu_t)=\langle \nabla F, \nabla \phi \rangle_{L^2(\rho)}.
\]
\end{proof}
In some typical cases, we can obtain an explicit description of $\nabla F$.
\begin{example}\label{exElasticity}
Let $\Gamma_0$, $\Gamma_1$ be two disjoint $d-1$-dimensional submanifolds in $\mathbb{R}^d$. We set $\Gamma:=\Gamma_0 \cup \Gamma_1$. Let $H_0^1(\mathbb{R}^d \backslash \Gamma_0; \mathbb{R}^d)$ be the Sobolev space on $\mathbb{R}^d \backslash \Gamma_0$ valued in $\mathbb{R}^d$. For $u \in H_0^1(\mathbb{R}^d \backslash \Gamma_0; \mathbb{R}^d)$, let $\sigma (u)$, $\ep(u)$ be the stress and strain tensors respectively, which are defined by
\begin{align*}
\ep(u)&=\frac{1}{2}(\nabla u + (\nabla u)^T),\\
\sigma(u)&:=2\mu \ep(u)+\lambda \text{tr}(\ep(u))I_d
\end{align*}
for positive constants $\mu$ and $\lambda$. For positive $\rho \in L^{\infty}(\mathbb{R}^d)$, consider the following system of elasticity for the displacement $u=u_{\rho}$:
\begin{align*}
\begin{cases}
-\mathrm{div}\left( b(\rho) \sigma(u) \right) + \delta u = f\ &\text{in}\ \mathbb{R}^d \backslash \Gamma,\\
\sigma(u)\bn=g\ &\text{on}\ \Gamma_1,\\
u=0\ &\text{on}\ \Gamma_0,
\end{cases}
\end{align*}
where $b \colon [0,\infty) \to (0,\infty)$ be a smooth positive function with a positive infimum, $f \in C^{\infty}(\mathbb{R}^d; \mathbb{R}^d)$, $g \colon \Gamma_1 \to \mathbb{R}^d$ is a smooth function and $\delta >0$. In this case, we can obtain an explicit formula of the first variation of the functional $\mathcal{J}$ by the adjoint method:
\begin{align*}
F(x) =b'(\rho)\sigma (u_{\rho}):\ep(u_{\rho})(x).
\end{align*}
If we suppose that the sufficient regularity of $\rho$, the state function $u_{\rho}$ is smooth. Therefore, by Proposition \ref{WassDens}, we have for $\phi \in C_0^{\infty}(\mathbb{R}^d)$ and the curve $\rho_t dx$ which satisfies the continuity equation with respect to $\nabla \phi$,
\[
\left( \frac{d}{dt} \right)_{t=0} \wt{\mathcal{J}}(\mu_t)=\langle \nabla^W \wt{\mathcal{J}}(\rho), \nabla \phi \rangle_{L^2(\rho)},
\]
where
\begin{align*}
\nabla^W\wt{\mathcal{J}}(\rho)&=\nabla \left(b'(\rho) \sigma (u_{\rho}):\ep(u_{\rho}) \right).
\end{align*}
This is supposed to provide the deformation of the material with the mass constraint.
\end{example}

\begin{remark}
Compared with the formulation in subsection \ref{Formulation1}, the gradient in the formulation in subsection \ref{Formulation2} can incorporates the change of topology and we do not need to consider the higher order differential. Moreover, we can relatively easily derive the sensitivity based on the sensitivity in the density method. However, we still need the additional penalization on the density in order to make the relaxed problem for $\wt{\cJ}$ a good approximation of the original problem defined on $\mathcal{O}_{ad}$.
\end{remark}

\ethics{CRediT authorship contribution statement}{F. Okazaki: Conceptualization, Formal Analysis, Investigation, Writing -- Original Draft. T. Yamada: Supervision.}

\begin{acknowledgement}
The authors are grateful to A. Nakayasu, K. Sakai and K. Matsushima for their salutary advice.
\end{acknowledgement}

\ethics{Competing Interests}{F. Okazaki has a received research grant from KAKEN-24K22827.\newline
The authors have no conflicts of interest to declare that are relevant to the content of this chapter.}

\end{document}